\documentclass[letterpaper, 11 pt, conference]{ieeeconf}
\usepackage{amsmath,amssymb,euscript,psfrag,latexsym,graphicx}
\usepackage{bbm,color,amstext,wasysym,subfig,cuted,mathtools}
\usepackage[normalem]{ulem}
\graphicspath{{./},{./figures/}}
\usepackage{dsfont,cite}

\newcommand{\mR}{{\mathbb R}}
\usepackage{algorithm}
\usepackage{algorithmic}
\usepackage[algo2e,linesnumbered,ruled]{algorithm2e}

\newcommand{\mP}{{\mathbb P}}
\newcommand{\bM}{{\mathbf M}}
\newcommand{\bv}{{\mathbf v}}

\newcommand{\bff}{{\mathbf f}}

\newcommand{\bw}{{\mathbf w}}
\newcommand{\bm}{{\mathbf m}}
\newcommand{\bD}{{\mathbf D}}
\newcommand{\mU}{{\mathbb U}}

\newcommand{\bs}{{\mathbf s}}

\newcommand{\bF}{{\mathbf F}}

\newcommand{\bx}{{\mathbf x}}

\newcommand{\bPsi}{{\boldsymbol \Psi}}

\newcommand{\bX}{{\mathbf X}}
\newcommand{\bY}{{\mathbf Y}}
\newcommand{\bg}{{\mathbf g}}

\newcommand{\bu}{{\mathbf u}}

\newcommand{\bc}{{\mathbf c}}

\newtheorem{theorem}{Theorem}
\newtheorem{definition}{Definition}
\newtheorem{lemma}{Lemma}
\newtheorem{assumption}{Assumption}
\newtheorem{remark}[theorem]{Remark}
\newtheorem{property}{Property}

\IEEEoverridecommandlockouts
\begin{document}

\title{ \bf Navigation with Probabilistic Safety Constraints: Convex Formulation}

\author{Joseph Moyalan, Yongxin Chen, and Umesh Vaidya
%UV will like to acknowledge the support of NSF 1932458, and DOE DE-OE-0000876. YC will like to acknowledge the support of the NSF under grant 1901599 and 1942523.
\thanks{J. Moyalan and U. Vaidya are with the Department of Mechanical Enginerring, Clemson University, Clemson SC, USA, jmoyala@clemson.edu, uvaidya@clemson.edu}

\thanks{Y. Chen is with the School of Aerospace Engineering, Georgia Institute of Technology, Atlanta, GA, USA, yongchen@gatech.edu}
}

\maketitle

%%%%%%%%%%%%%%%%%%%%%%%%%%%%%%%%%%%%%%%
%%%%%%%%%%%%   Abstract   %%%%%%%%%%%%%
%%%%%%%%%%%%%%%%%%%%%%%%%%%%%%%%%%%%%%%

\begin{abstract}
We consider the problem of navigation with safety constraints. The safety constraints are probabilistic, where a given set is assigned a degree of safety, a number between zero and one, with zero being safe and one being unsafe. The deterministic unsafe set will arise as a particular case of the proposed probabilistic description of safety. We provide a convex formulation to the navigation problem with probabilistic safety constraints. The convex formulation is made possible by lifting the navigation problem in the dual space of density using linear transfer operator theory methods involving Perron-Frobenius and Koopman operators. The convex formulation leads to an infinite-dimensional feasibility problem for probabilistic safety verification and control design. The finite-dimensional approximation of the optimization problem relies on the data-driven approximation of the linear transfer operator.

% We consider stochastic  safety  control  synthesis  and  stochastic safety verification problems for  control-affine  systems. The main goal is to drive a nonlinear system from an initial set to the desired set while avoiding some unsafe set whose information is stochastic in nature. We make use of Lyapunov density function to convert the stochastic navigation problem into a convex optimization problem while making use of probability density function of the unsafe sets. The radial basis functions are utilised along with linear operator theory to solve the control synthesis problem and safety verification. Numerical examples are presented to showcase the proficiency of the proposed framework.
\end{abstract}

% \begin{IEEEkeywords}
% Stochastic optimal control, Convex optimization, Linear operators, 
% \end{IEEEkeywords}

\section{Introduction}
The obstacle avoidance problem is one of the necessities for robots and unmanned vehicles operating in a structured or unstructured environment. The classical approach to obstacles modeling and avoidance is deterministic in nature \cite{fiorini1998motion,hongzhe2021convex,hongzhe2021obs}. However, increasingly there is a shift in focus for obstacle modeling and avoidance to a probabilistic or stochastic setting. This shift in focus is mainly driven by navigation applications in unstructured or off-road terrain where environmental uncertainty is a norm.

The authors of \cite{du2011probabilistic} investigate the probabilistic collision checking between uncertain configurations for two objects, which is referred to as collision chance constraints. They provide a probability of collision that accounts for both robot and obstacle uncertainty. 
% However, the motion planning problem addressed in \cite{du2011probabilistic} with uncertain obstacles is non-convex, which can be a bottleneck for many high dimensional systems. 
In \cite{majumdar2020should}, an investigation of risk metrics, i.e., mapping from a random variable corresponding to costs to a real number, for different probabilistic applications is done. The authors of \cite{majumdar2020should} provide various axioms that help in the selection of risk metrics which is utilized in the cost function. However, finding such risk metrics for everyday applications is, in general, not trivial. In \cite{fan2021step}, the authors tackle the problem of autonomous traverse in an unknown environment by efficient risk and constraint-aware kinodynamic motion planning using sequential quadratic programming-based (SQP) model predictive control (MPC). They propose a tail risk assessment using the Conditional Value-at-Risk (CVaR). However, the problem is formulated in a nonlinear setting with highly non-convex constraints on the robot's motion. In \cite{hacohen2019probability}, a new probabilistic navigation function (PNF) has been proposed to work in a stochastic environment. The proposed PNF provides an analytic proof for convergence and considers both the geometries and the location probability functions without inflating the ambient space dimension. But so far, all the formulations involving navigation with probabilistic descriptions of the obstacle set are formulated as non-convex problems. This non-convex nature of the problem is one of the main challenges in developing systematic analysis and synthesis methods. 

% with highly non-convex constraints, which proves to be challenging to solve with an increase in dimension and complexity of the system dynamics.\\

In this paper, we provide a probabilistic description of safety using a probability measure. The probability measure is used to assign a varying degree of safety for any given set based on the value of the probability measure on the set. Larger the value of probability measure on a set more unsafe the set. Furthermore, the deterministic description of an unsafe set will arise as a particular case of the proposed probabilistic description. We consider the problem of navigation with probabilistic safety constraints. 
The main contribution is in providing convex formulation to the verification and control synthesis for the navigation problem. The convex formulation is made possible by lifting the navigation problem in the space of density using methods from linear operator theory involving Perron-Frobenius and Koopman operators \cite{Lasota}. This convex formulation builds on our past work on providing convex formulation to navigation, stabilization, and optimal control problems \cite{vaidya2018optimal,huang2020convex,hongzhe2021convex,moyalan2021sum}.
The convex formulation leads to an infinite-dimensional convex feasibility problem for the probabilistic safety verification and synthesis problems. The infinite-dimensional feasibility problem is approximated using a data-driven method developed for the finite-dimensional approximation of the linear transfer operator.

% unsafe sets (obstacles) in a stochastic setting with the help of probability measure. This probability measure is an integral function of probability density function. It quantifies the probability of a set being unsafe. We consider a scenario where value of probability measure is nonzero everywhere but the degree of hazard of a set is determined by its probability density function. The main contribution of this paper is two-fold. First, we provide a safety verification of the navigation problem with stochastic obstacles in a convex setting. Secondly, we formulate control synthesis for stochastic safety navigation problem as a convex optimization problem. We make use of density function to transform the safety verification problem into infinite dimensional convex optimization problem in dual density space. We lift the navigation problem in the dual space of density using
% methods from linear transfer operator theory involving
% Perron-Frobenius and Koopman operators \cite{vaidya2008lyapunov,raghunathan2013optimal,huang2020convex}. The lifting is done using some observables called as lifting functions.  These  lifting  functions  forms  the  basis  for  the lifted density space.\\

The rest of the paper is structured as follows. In Section II, we provide a brief introduction to our framework’s necessary preliminaries.  Section  III  consists of a probabilistic description of unsafe sets, and the main theoretical results are described in section IV.  In  Section  V,  we develop the computational framework based on the linear operator framework.  This is followed by simulation examples in Section VI and a conclusion in Section VII.

\section{Preliminaries and Notations}\label{section_prelim}
\noindent {\it Notations}: $\mR^n$ denotes the $n$ dimensional Euclidean space  and $\mR^n_{\geq 0}$ is the positive orthant. Given $\bX\subseteq \mR^n$ and $\bY\subseteq \mR^m$, let ${\cal L}_1(\bX,\bY), {\cal L}_\infty(\bX,\bY)$, and ${\cal C}^k(\bX,\bY)$ denote the space of all real valued integrable functions, essentially bounded functions, and space of $k$ times continuously  differentiable functions mapping from $\bX$ to $\bY$ respectively. If the space $\bY$ is not specified then it is understood that the underlying space is $\mR$. ${\cal B}(\bX)$ denotes the Borel $\sigma$-algebra on $\bX$ and ${\cal M}(\bX)$ is the vector space of real-valued measure on ${\cal B}(\bX)$. $\bs_t(\bx)$ denotes the solution of dynamical system $\dot \bx={\bf F}(\bx)$ starting from initial condition $\bx$. $\bX_0$, $\bX_r$ are assumed to be the initial set and final terminal set respectively. With no loss of generality we will assume that $\bX_r=\{0\}$. Let ${\cal N}_\delta$ be the neighborhood of $\bX_r$ for some fixed $\delta>0$ i.e., $\bX_r\subset {\cal N}_\delta$. Let $\bX_1:=\bX\setminus {\cal N}_
\delta$. Let $h_0(\bx)$ be the probability density function corresponding to the probability measure $\mu_0$ with support on set $\bX_0$.

\subsection{Perron-Frobenius and Koopman Operator}
Consider a dynamical system of the form
\begin{eqnarray}
\dot \bx={\bF}(\bx),\;\;\;\bx\in \bX\subseteq \mathbb{R}^n\label{sys}.
\end{eqnarray}
where the vector field is assumed to be $\bF(\bx)\in {\cal C}^1(\bX,\mR^n)$. 
There are two different ways of lifting the finite dimensional nonlinear dynamics from state space to infinite dimension  space of functions namely using Koopman and Perron-Frobenius operators. The definitions of these operators along with the infinitesimal generators of these operators are defined as follows \cite{Lasota}.
\begin{definition}[Koopman Operator]  $\mathbb{U}_t :{\cal L}_\infty(\bX)\to {\cal L}_\infty(\bX)$ for dynamical system~\eqref{sys} is defined as 
\begin{eqnarray}[\mathbb{U}_t \varphi](\bx)=\varphi(\bs_t(\bx)). \label{koopman_operator}
\end{eqnarray}
The infinitesimal generator for the Koopman operator is given by
\begin{eqnarray}
\lim_{t\to 0}\frac{(\mathbb{U}_t-I)\varphi}{t}=\bF(\bx)\cdot \nabla \varphi(\bx)=:{\cal K}_{\bF} \varphi,\;\;t\geq 0. \label{K_generator}
\end{eqnarray}
\end{definition}

\begin{definition} [Perron-Frobenius Operator] $\mathbb{P}_t:{\cal L}_1(\bX)\to {\cal L}_1(\bX)$ for dynamical system~\eqref{sys} is defined as 
\begin{eqnarray}[\mathbb{P}_t \psi](\bx)=\psi(\bs_{-t}(\bx))\left|\frac{\partial \bs_{-t}(\bx) }{\partial \bx}\right|, \label{PF_operator}
\end{eqnarray}
where $\left|\cdot \right|$ stands for the determinant. The infinitesimal generator for the P-F operator is given by 
\begin{eqnarray}
\lim_{t\to 0}\frac{(\mathbb{P}_t-I)\psi}{t}&=-\nabla \cdot (\bF(\bx) \psi(\bx)) \nonumber \\
&=: {\cal P}_{\bF}\psi,\;\;t\geq 0. \label{PF_generator}
\end{eqnarray}
\end{definition}
% \chen{do we want to emphasize the following duality? We don't use it in this paper.} \hc{But, in my opinion, this is still useful since we are using PF operator in our algorithm as shown in~\eqref{eq:PF operators}.}\chen{the way we calculate P-F operator is different, which doesn't depend on this duality} {\color{red}I think we should keep this.. as I am making reference to this in the introduction..and will also be complete... without this duality the connection between the two is superficial and mainly used for computation..pulled out of nowhere..}
These two operators are dual to each other where the duality is expressed as follows.
\begin{eqnarray*}
\int_{\bX}[\mU_t \varphi](\bx)\psi(\bx)d\bx=
\int_{\bX}[\mathbb{P}_t \psi](\bx)\varphi(\bx)d\bx.
\end{eqnarray*}
\begin{property}\label{property}
These two operators enjoy positivity and Markov properties which are used in the finite dimension approximation of these operators. 
\begin{enumerate}
    \item Positivity: The P-F and Koopman operators are positive operators i.e., for any $0\leq\varphi(\bx)\in {\cal L}_\infty(\bX)$ and $0\leq \psi(\bx)\in {\cal L}_1(\bX)$, we have
    
    \begin{equation}\label{positive_prop}
        [\mathbb{P}_t\psi](\bx)\geq 0,\;\;\;\;[\mathbb{U}_t\varphi](\bx)\geq 0,\;\;\;\forall t\geq 0.
    \end{equation}
    \item Markov Property:  The P-F operator satisfies Markov property i.e.,
\begin{equation}\label{Markov_prop}
    \int_\bX [\mathbb{P}_t\psi](\bx)d\bx=\int_\bX\psi(\bx)d\bx.
\end{equation}
\end{enumerate}
\end{property}
\vspace{0.1in}
\begin{assumption}\label{assume_localstability} We assume that $\bx=0$ is locally stable equilibrium point for the system (\ref{sys}) with local domain of attraction denoted by ${\cal N}_\delta$ for a fixed $\delta>0$. We let $\bX_1:=\bX\setminus {\cal N}_\delta$
\end{assumption}

\begin{definition}[\small Almost everywhere (a.e.) uniform stability] \label{def_aeuniformstable} The equilibrium point is said to be a.e. uniform  stable w.r.t. measure $\mu_0\in {\cal M}(\bX)$ if for any given $\epsilon$, there exists a time $T(\epsilon)$ such that
\begin{eqnarray}
\int_{T(\epsilon)}^\infty \mu_0 (B_t)dt<\epsilon,\label{eq_aeunifrom}
\end{eqnarray}
where $B_t:=\{\bx \in \bX_1: \bs_t(\bx)\in B\}$ for every set $B\in{\cal B}(\bX_1)$.
\end{definition}

% \begin{assumption}\label{assumption_mu}
% We assume that the measure $\mu\in {\cal M}(\bX)$ is equivalent to Lebesgue with Radon–Nikodym derivative $h$ i.e.,  $\frac{d\mu}{d\bx}=h(\bx)>0$ and $h\in {\cal L}_1(\bX,\mR_{> 0})\cap {\cal C}^1(\bX)$.
% \end{assumption}

The following theorem is from \cite[Theorem 13]{rajaram2010stability} providing necessary and sufficient condition for a.e. uniform stability. 

\begin{theorem}\label{theorem_necc_suff}
The equilibrium point $\bx=0$ for system (\ref{sys}) satisfying Assumption \ref{assume_localstability} is a.e. uniformly stable w.r.t. measure $\mu_0$ if and only if there exists a  function $\rho(\bx)\in{\cal C}^1(\bX\setminus \{0\},\mR_{\geq 0})\cap {\cal L}_1(\bX_1)$  and satisfies
\begin{eqnarray}
\nabla\cdot ({\bf F}\rho )=h_0.\label{steady_pde1}
\end{eqnarray}
where $h_0\in {\cal L}_1(\bX,\mR_{\geq 0})$ is assumed to be the density function corresponding to the measure $\mu_0$. 
\end{theorem}

\subsection{Data-Driven Approximation: Naturally Structured Dynamic Mode Decomposition}\label{section_nsdmd}
Naturally structured dynamic mode decomposition (NSDMD) is a modification of Extended Dynamic Mode Decomposition (EDMD) algorithm \cite{williams2015data}, one of the popular algorithms for Koopman approximation from data. The modifications are introduced to incorporate the natural properties of these operators namely positivity and Markov. For the continuous-time dynamical system (\ref{sys}), consider snapshots of data set obtained as time-series data from single or multiple trajectories
\begin{eqnarray}
{\mathcal X}= [\bx_1,\bx_2,\ldots,\bx_M],\;\;\;\;{\cal Y} = [\mathbf{y}_1,\mathbf{y}_2,\ldots,\mathbf{y}_M] ,\label{data}
\end{eqnarray}
where $\bx_i\in \bX$ and $\mathbf{y}_i\in \bX$. The pair of data sets are assumed to be two consecutive snapshots i.e., $\mathbf{y}_i=\bs_{\Delta t}(\bx_i)$, where $\bs_{\Delta t}$ is solution of (\ref{sys}). Let ${\bPsi}=[\psi_1,\ldots,\psi_N]^\top$ be the choice of basis functions.
The popular Extended Dynamic Mode Decomposition (EDMD) algorithm provides the finite-dimensional approximation of the Koopman operator as the solution of the following least square problem. 

\begin{equation}\label{edmd_op}
\min\limits_{\bf K}\parallel {\bf G}{\bf K}-{\bf A}\parallel_F,
\end{equation}
where,
{\small
\begin{eqnarray}\label{edmd1}
{\bf G}=\frac{1}{M}\sum_{m=1}^M \bPsi({\bx}_m) \bPsi({\bx}_m)^\top,\\
{\bf A}=\frac{1}{M}\sum_{m=1}^M \bPsi({\bx}_m) \bPsi({\mathbf y}_m)^\top,
\end{eqnarray}}
with ${\bf K},{\bf G},{\bf A}\in\mathbb{R}^{N\times N}$, $\|\cdot\|_F$ stands for Frobenius norm. The above least square problem admits an analytical solution 
\begin{eqnarray}
{\bf K}_{EDMD}=\bf{G}^\dagger \bf{A}\label{edmd_formula}.
\end{eqnarray}
Convergence results for EDMD algorithms in the limit as the number of data points and basis functions go to infinity are provided in \cite{korda2018convergence,klus2020eigendecompositions}.
In this paper, we work with Gaussian Radial Basis Function (RBF) for the finite-dimensional approximation of the linear operators.
Under the assumption that the basis functions are positive, like the Gaussian RBF, the NSDMD algorithm propose following convex optimization problem for the approximation of the Koopman operator that preserves positivity and Markov property in Property~\ref{property}. 
\begin{eqnarray}\label{nsdmd}
&\min\limits_{\hat{\bf P}}\parallel \hat{\bf G}{\hat{\bf P}}-\hat{\bf A}\parallel_F\\\nonumber
\text{s.t.} \;\;
& [\hat{\bf P}]_{ij}\geq 0,\;\;\;\hat {\bf P}\mathds{1} = \mathds{1},
\end{eqnarray}
where,
\begin{eqnarray}\hat{\bf G}={\bf G}{\bf \Lambda}^{-1},\;\;\hat{\bf A}={\bf A}{\bf \Lambda}^{-1},\;\;\&\;\;{\bf \Lambda}=\int_\bX {\bPsi}{\bPsi}^\top d\bx\label{hatGA},
\end{eqnarray}
with $\bf G$ and $\bf A$ are as defined in \ref{edmd1} and $\mathds{1}$ is a vector of all ones. All the matrices in Eq. (\ref{hatGA}) are pre-computed from the data. In fact, since the basis functions are assumed to be Gaussian RBF, the constant ${\bf \Lambda}$ matrix can be computed explicitly as 
\[\Lambda_{i,j} =(\frac{\pi\sigma^2}{2})^{n/2} \exp^\frac{-\lVert \bc_i-\bc_j\rVert^2}{2\sigma^2}, i,j=1,2,\ldots,N,\]
where $\bc_i,\bc_j$ are the centers of the $\psi_i$ and $\psi_j$ Gaussian RBFs respectively. 
The constraints in (\ref{nsdmd}) ensure that finite-dimensional approximation preserves the positivity property and Markov property respectively. The approximation for the P-F operator and its generator are obtained as the solution of the optimization problem (\ref{nsdmd}) as 
\begin{eqnarray}
\mathbb{P}_{\Delta t}\approx \hat{\bf P}^\top=: {\bf P},\;\;\;\;\;{\cal P}_{\bF}\approx \frac{\hat{\bf P}^\top-{\bf I}}{\Delta t}=:\bM.\label{PF_approximation}
\end{eqnarray}

\section{Probabilistic Description of Unsafe Set}
% Consider an autonomous dynamical system of the form.
% \begin{eqnarray}
% \dot \bx={\bf F}(\bx),\;\;\;\;\;\bx\in \bX\subset \mR^n \label{sys}
% \end{eqnarray}
In this paper, we consider a probabilistic description of the unsafe set defined as follows. Let $p(\bx)$ is the probability density function  describing the unsafe set and $\mu_p\in {\cal M}(\bX)$ is the associated probability measure i.e., $d\mu_p(\bx)=p(\bx) d\bx$.
Let $A\in {\cal B}(\bX)$, then the probability that the set $A$ is unsafe  is defined using $p(\bx)$ as 
\begin{align}
    {\rm Prob}(A \;{\rm is\; Unsafe})=\int_A p(\bx)d\bx=:\mu_p(A).
\end{align}
The above definition include the deterministic description of unsafe set as the special case.  In particular, if $\bX_u$ is an unsafe set then it can described using the following definition of the uniform probability density function
\begin{align}
    p(\bx)=\frac{1}{m(\bX_u)}\mathds{1}_{\bX_u}(\bx)
\end{align}
where $m(\cdot)$ is the Lebesgue measure and $\mathds{1}_{\bX_u}$ is the indicator function of the set $\bX_u$. 
So by allowing $p(\bx)$ to be nonuniform probability density function we can consider cases where different regions of the state space $\bX$ have varying degree of safety. 
Note that when $p(\bx)$ is supported on the entire set $\bX$, then the entire region of the state space $\bX$ is potentially hazardous; however, the degree of hazard is determined by the probability density function $p(\bx)$. In particular, if 
\[\mu_p(A_1)<\mu_p(A_2)\]
where $A_i\in {\cal B}(\bX)$, then the region $A_2$ is more hazardous than region $A_1$. Strictly speaking, in the above described probabilistic setting, it may not be appropriate to use the terminology of a safe or unsafe set as it has a connotation of carrying binary information. Consider the case when $p(\bx)$ is supported on the entire space $\bX$, then potentially the entire space $\bX$ is hazardous.  However, with nonuniform probability density $p(\bx)$, some regions of $\bX$ are more favorable to navigate and hence potentially less hazardous than others. The objective then is to navigate through regions of $\bX$ where $\mu_p$ is small. With some abuse of terminology, we will continue to refer to the sets in $\bX$ as safe and unsafe where the degree of safety is characterized by $p(\bx)$.  We defined the probability of collision as follows.
\begin{definition} The probability of collision with the unsafe set under the system dynamics (\ref{sys}) with the initial condition distributed w.r.t. probability measure $\mu_0$ is given by
\begin{align}
\int_{0}^\infty \int_{\bX} p(\bs_t(\bx))d\mu_0(\bx)dt\label{def_collision}
\end{align}
\end{definition}
\vspace{0.1in}
% The
% probability of navigating the system dynamics (\ref{sys}) through an unsafe set  with initial condition distributed w.r.t. probability measure $\mu_0$ is given by
% \begin{align}
% \int_{0}^\infty \int_{\bX} p(\bs_t(\bx))d\mu_0(\bx)dt
% \end{align}
The objective is to minimize the above collision probability while navigating the system dynamics from some initial set $\bX_0$, supported on measure $\mu_0$ to some final terminal set $\bX_r$ asymptotically.

The reason for using the formula (\ref{def_collision}) in the definition of collision probability stems from the following Lemma.  The following Lemma essentially points to the fact that the formula (\ref{def_collision}) is a measure of occupancy in the unsafe region, and this connection is evident in the case of a deterministic unsafe set. 
\vspace{0.1in}
\begin{lemma}
Let $ p(\bx)=\frac{1}{m(\bX_u)}\mathds{1}_{\bX_u}(\bx)$. If 
\begin{align}
\int_{0}^\infty \int_{\bX} p(\bs_t(\bx))d\mu_0(\bx)dt=0
\end{align} then 
\begin{align}
    \int_{\bX_1} \mathds{1}_{\bX_u}(\bs_t(\bx))h_0(\bx)d\bx=0,\;\;\;\forall t\geq 0.\label{contra}
\end{align}
i.e., the amount of time system trajectories spend in the unsafe set $\bX_u$ starting from the positive measure set of initial condition corresponding to the initial set, $\bX_0$, with density $h_0(\bx)$ is equal to zero. 
\label{lemma:time_unsafeset}

\begin{proof}
Proof by contradiction. 
Assume (\ref{contra}) is not true, i.e., there exists some time $t_0$ for which 
\[\int_{\bX_1} \mathds{1}_{\bX_u}(\bs_{t_0}(\bx))h_0d\bx=\int_{\bX_1}[ \mU_{t_0}\mathds{1}_{\bX_u}](\bx)h_0d\bx>0\]
Then using the continuity property of the Koopman semi-group, we know there exists a $\Delta$ such that 
\[\int_{t_0}^{t_0+\Delta}\int_{\bX_1}[ \mU_{t_0}\mathds{1}_{\bX_u}]((\bx))h_0(\bx)d\bx dt>0.\]
We have 
\begin{equation}
\begin{aligned}
    0<\int_{t_0}^{t_0+\Delta}\int_{\bX_1}[ \mU_{t_0}\mathds{1}_{\bX_u}](\bx)h_0(\bx)d\bx dt\leq\\ \int_{0}^{\infty}\int_{\bX_1}[ \mU_{t_0}\mathds{1}_{\bX_u}](\bx)h_0(\bx)d\bx dt\nonumber\\=\int_{0}^\infty \int_{\bX} p(\bs_t(\bx))d\mu_0(\bx)dt=0\nonumber.
\end{aligned}
\end{equation}
\end{proof}
\end{lemma}
So by minimizing the quantity in (\ref{def_collision}) for a general non-uniform probability density function $p(\bx)$, we are essentially trying to minimize the occupancy in the unsafe set.

\section{Main Results}
The main results of this paper are  presented in the following two subsections and  provide for a  convex formulation to the probabilistic safety verification and control synthesis problems. 
\subsection{Probabilistic Safety Verification}
In this section, we show that the probabilistic safety navigation problem can be  verified convexly. In particular, the stochastic safety navigation problem can be written as convex optimization problem. The stochastic safety navigation problem can be stated as follows.
\begin{definition}[Stochastic Safety Navigation Problem]\label{def_verification} The problem consists of  navigating almost every (a.e.) (w.r.t. Lebesgue measure) trajectories of the system 
\begin{align}
    \dot\bx={\bf F}(\bx)\label{sys_dyn}
\end{align}
starting from some initial set, $\bX_0$, to some final set $\bX_r$ asymptotically, while ensuring that the probability of collision with the obstacle set should be less than or equal to $\gamma$ i.e., 
\begin{align}
\int_{0}^\infty \int_{\bX_1} p(\bs_t(\bx))d\mu_0(\bx)dt\leq \gamma.\label{prob_gamma}
\end{align}
\end{definition}
The main results of this section provides convex formulation to the stochastic safety verification problem as defined in Definition \ref{def_verification}. Before that we make following assumption on the system dynamics (\ref{sys_dyn}).
\begin{assumption}\label{assume_local}
We assume that the final destination set is locally stable with local domain of attraction $N_\delta$ for a fixed $\delta>0$.  
\end{assumption}
\begin{theorem}\label{theorem_verification}
Under Assumption \ref{assume_local}, the stochastic safety navigation problem (Definition \ref{def_verification}) can be verified convexly if there exists a density function $\rho(\bx)\in{\cal L}_1(\bX_1)\cap {\cal C}^1(\bX_1,\mR_{\geq 0})$ and  satisfies 
\begin{align}
&\nabla \cdot ({\bf F}(\bx) \rho(\bx))=h_0(\bx),\;\;{\rm a.e.}\;\bx\in \bX_1\label{feas11}\\
&\int_{\bX_1} p(\bx)\rho(\bx)d\bx\leq \gamma\label{feas1}
\end{align}
\end{theorem}
\begin{proof}
Following the results of Theorem \ref{theorem_necc_suff}, since $\rho(\bx)$ satisfies (\ref{feas11}) it follows that the $\bX_r=\{0\}$ is a.e. uniform stable w.r.t. measure $d \mu_0(\bx)=h_0(\bx)d\bx$ i.e., a.e. initial condition from set $\bX_0$ supported on measure $\mu_0$ will be attracted to the final set $\bX_r$. Furthermore, the density function $\rho(\bx)$ that satisfies (\ref{feas11}) can be expressed using the following integral formula
\begin{align}
    \rho(\bx)=\int_0^\infty [\mP_t h_0](\bx)dt \label{integral}
\end{align}
We next show that (\ref{feas1}) implies (\ref{prob_gamma}) i.e., the probability of collision is less than or equal to $\gamma$. We have,
{\small
\begin{align}
\int_0^\infty \int_{\bX_1} p(\bs_t(\bx))d\mu_0(\bx)dt= 
\int_0^\infty \int_{\bX_1} [\mU_t p](\bx)d\mu_0(\bx) dt \nonumber
\end{align}}
Now using duality between the Koopman and P-F operator and the fact that $\frac{d\mu_0}{d\bx}=h_0$, we obtain
{\small
\begin{align}
\int_0^\infty \int_{\bX_1} [\mU_t p](\bx)d\mu_0(\bx) dt= 
\int_0^\infty\int_{\bX_1}p(\bx)[\mP_t h_0](\bx)d\bx dt \nonumber
\end{align}}
Exchanging the spatial and temporal integration and using the integral formula for the $\rho(\bx)$ in (\ref{integral}) we obtain
\[\int_0^\infty\int_{\bX_1}p(\bx)[\mP_t h_0](\bx)d\bx dt=\int_{\bX_1}p(\bx)\rho(\bx)d\bx\]

\end{proof}

\subsection{Convex Approach to Control Synthesis with Stochastic Safety Constraints}
In this section, we formulate control synthesis  for stochastic safety navigation problem as a convex optimization problem. The convex control synthesis is possible under the assumption that control system is affine in input.  Consider a control affine dynamical system of the form
\begin{align}
    \dot \bx=\bff(\bx)+\bg(\bx)u\label{sys_control}
\end{align}
where $\bu\in \mathbb{R}$ is the control input. For the simplicity of presentation we restrict the discussion to the case of single input. It is straight forward to extend this framework to multi-input case. The objective is to design the control input so that almost all trajectories of the system (\ref{sys_control}) starting from the initial set $\bX_0$ to the final set $\bX_r$ minimizes the probability of collision with the obstacle set to less than equal to $\gamma$ (Eq. (\ref{prob_gamma}))
Following assumption is made on the control synthesis problem for safety navigation.
\begin{assumption}\label{assume1}
We assume that the control for stochastic safety navigation is feedback in nature i.e., $u=k(\bx)\in {\cal C}^1(\bX)$ and that the final destination set $\bX_r$ is locally stable with local domain of attraction ${\cal N}_\delta$ for a fixed $\delta>0$.  
\end{assumption}
Following is the main results of this section. 
\begin{theorem}\label{theorem2}Under Assumption \ref{assume1}, the control synthesis problem for the stochastic safety navigation with norm constraints on the control input of the form, $|u|\leq M$, can be formulated as following convex feasibility problem in terms of variable $\rho \in{\cal L}_1(\bX_1)\cap {\cal C}^1(\bX_1,\mR_{\geq 0})$ and $\bar \rho\in{\cal C}^1(\bX_1,\mR) $.
\begin{eqnarray}
\nabla\cdot(\bff \rho+\bg \bar \rho)=h_0,\;\;{\rm a.e.}\; \bx\in \bX_1\label{feas21}\\
\int_{\bX_1}p(\bx)\rho(\bx)d\bx\leq \gamma\label{feas22}\\
|\bar\rho(\bx)|\leq M \rho(\bx)  \label{feas2}
\end{eqnarray}
The control input is given by 
\begin{align}
    u=k(\bx)=\frac{\bar \rho(\bx)}{\rho(\bx)}\label{feedback_control}
\end{align}

\end{theorem}
\begin{proof}
The proof of this Theorem follows along similar lines to the proof of Theorem \ref{theorem_verification} applied to feedback control system $\dot \bx={\bf F}(\bx):=\bff (\bx)+\bg(\bx)k(\bx)$. This gives us 
\[\nabla\cdot((\bff+\bg k)\rho)=h_0\]
The above equation is bilinear in the two unknowns $\rho$ and $k$ and can be converted to linear equation in terms of new unknown variables  and $\rho$ and $\bar \rho$ defined as  $\bar \rho(\bx):=k(\bx)\rho(\bx)$. After the change of variables from $(\rho,k)\to (\rho,\bar \rho)$, we obtain (\ref{feas21}). The derivation of  (\ref{feas22}) follows along the same lines of corresponding derivation in Theorem \ref{theorem_verification}. The condition (\ref{feas2}) ensures that the formula for the control input in (\ref{feedback_control}) is well defined and that the constraints on the control input i.e., $|u|\leq M$ are satisfied. 
\end{proof}

\section{Computational Framework}
For the finite dimensional approximation of the infinite dimensional  convex feasibility problems (\ref{feas1}) and (\ref{feas2}), we need to construct the approximation of the generator corresponding to vector field $\bf f$ and $\bg$ i.e., $\nabla\cdot ({\bf f}\rho)$ and $\nabla\cdot(\bg\bar \rho)$. Following notations from (\ref{PF_generator}) and (\ref{PF_approximation}), let the approximations of the two generators corresponding to vector fields $\bff$ and $\bg$ be denoted by
\begin{align}
 {\cal P}_{\bff}\approx \bM_0,\;\;\;{\cal P}_{\bg}\approx \bM_1
\end{align}
We use NSDMD algorithm outlined in Section \ref{section_nsdmd} for this approximation with 
$\bPsi(\bx)=(\psi_1,\ldots, \psi_N)^\top$ as the basis functions. Let $h_0(\bx)$, $p(\bx)$, $ \rho(\bx)$, and $\bar \rho(\bx)$ be expressed in terms of the basis function 
\begin{align}
&h_0(\bx)=\bPsi^\top {\bf m},\;\rho(\bx)\approx \bPsi^\top{\bf v},\;\bar\rho(\bx)\approx \bPsi^\top {\bf w}\nonumber\\
&p(\bx)\approx \bPsi^\top \bu .\label{approx}
\end{align}
Using (\ref{approx}), we can also write
\begin{align}
    \int_{\bX_1}p(\bx)\rho(\bx)d\bx\approx \bu^\top \left[\int_{\bX_1}\bPsi(\bx)\bPsi^\top(\bx)d\bx \right]\bv\\=\bu^\top {\bf D} \bv
\end{align}
where ${\bf D}:=\int_{\bX_1}\bPsi(\bx)\bPsi^\top(\bx)d\bx$.
% {\color{red}
% \begin{assumption}\label{assume_rbf}
% We assume that the basis functions, $\psi_k(\bx)$ for $k=1,\ldots,N$ are positive and let \[\bPsi(\bx)=[\psi_1(\bx),\ldots,\psi_N(\bx)]^\top.\]
% \end{assumption}}
\begin{assumption}\label{assume_gaussian}
We assume that all the basis functions are positive and are linearly independent. 
\end{assumption}
\begin{remark}
\label{remark_positivebasis}
In this paper, we use Gaussian RBF to obtain all the simulation results i.e.,
$
\psi_k(\bx)=\exp^{-\frac{\parallel \bx-{\bf c}_k\parallel^2}{2\sigma^2}}$
where $\bc_k$ is the center of the $k^{th}$ Gaussian RBF. 
\end{remark}
Using Assumption \ref{assume_gaussian}, the finite dimensional approximation of the convex synthesis problem from Theorem \ref{theorem2} can be written as 
\begin{align}
    -\bM_0\bv-\bM_1\bw=\bm\label{ll1}\\
    \bu^\top \bD \bv\leq \gamma\\
    \bv\geq 0,\;\;\;|\bw|\leq M \bv\label{ll3}
\end{align}
where the inequalities are assumed to the element wise. The feedback control input is then obtained using the following formula
\begin{align}
    u=k(\bx)=\frac{\bPsi^\top(\bx) \bw}{\bPsi^\top(\bx) \bv}
\end{align}

\section{Simulation Results}
All the simulation results in this paper are obtained using Gaussian RBF. The value of $\gamma$ is taken to be $0.1$ for all the examples. The following rules of thumb are abided in selecting centers and $\sigma$ parameters for the Gaussian RBF. The RBF centers are chosen to be uniformly distributed in the state space at a distance of $d$. The $\sigma$ for the Gaussian RBF is chosen such that $d\leq 3\sigma\leq 1.5 d$. All the simulation results are performed using MATLAB on a desktop computer with 16GB RAM.  The optimization problem is solved using CVX. We present simulation results involving two examples and in both the examples we design local controller around a small neighbourhood of $\bX_r$ to stabilize the equilibrium point. Let  us consider this small neighbourhood to be $\mathbf{X}_L $ which is defined as follows:
\begin{align}
    \mathbf{X}_L = \{\bx \in \mathbf{X} : |\bx - \mathbf{X}_r | \le \epsilon , \epsilon >0 \}\nonumber
\end{align}
Here, the value of $\epsilon = 0.002$. We   design   an   optimal   linear   quadratic   regulator   (LQR) controller for $\mathbf{X}_L$ from the linearized dynamics of the model to guarantee the convergence in $\mathbf{X}_L$.
\\
\noindent{\textbf{Example 1.}} We consider the problem of navigating a boat on a river \cite{rantzer2004analysis}. The dynamics of the  model is given by:
\begin{equation*}
    \dot{x_1} = 1+0.125\cos(0.5x_1)-0.125\sin(0.5x_2); \;\; \dot{x_2} = u;
\end{equation*}
We define the sets to be:
\begin{itemize}
\small
    \item $\mathbf{X} \triangleq \{\bx \in \mathds{R}^2 : 0 \le x_1 \le 10,0 \le x_2 \le 10 \}$
    \item $\mathbf{X}_0 \triangleq \{\bx \in \mathbf{X} : 0 \le x_1 \le 1,5 \le x_2 \le 7 \}$
    \item $\mathbf{X}_u \triangleq \{\bx \in \mathbf{X} : (2\sin(x_1)+6-x_2)(x_2-2\sin(x_1)-4)\le 0 \;\;or\;\; x_1 \le 0\}$
    \item $\mathbf{X}_r \triangleq [9.7,3.8]^T$.
\end{itemize}
We have used $2500$ radial basis functions as lifting functions and $\Delta t = 0.01$. 
Fig.~\ref{exp1_tract} displays the open and closed loop trajectories navigating the river dynamics. We can see from this plot that the controller obtained using the finite-dimensional approximation of the feasibility problem (\ref{ll1})-(\ref{ll3}) is able to avoid unsafe set successfully.   Fig.~\ref{exp1_rho} displays the color plots of the density function $\rho(\bx)$ which clearly displays the differences in the density values between safe and unsafe sets. In particular, the value of the density function on the unsafe set is equal to zero.\\ 
\begin{figure}[ht]
\centering
\includegraphics[width=0.9\linewidth]{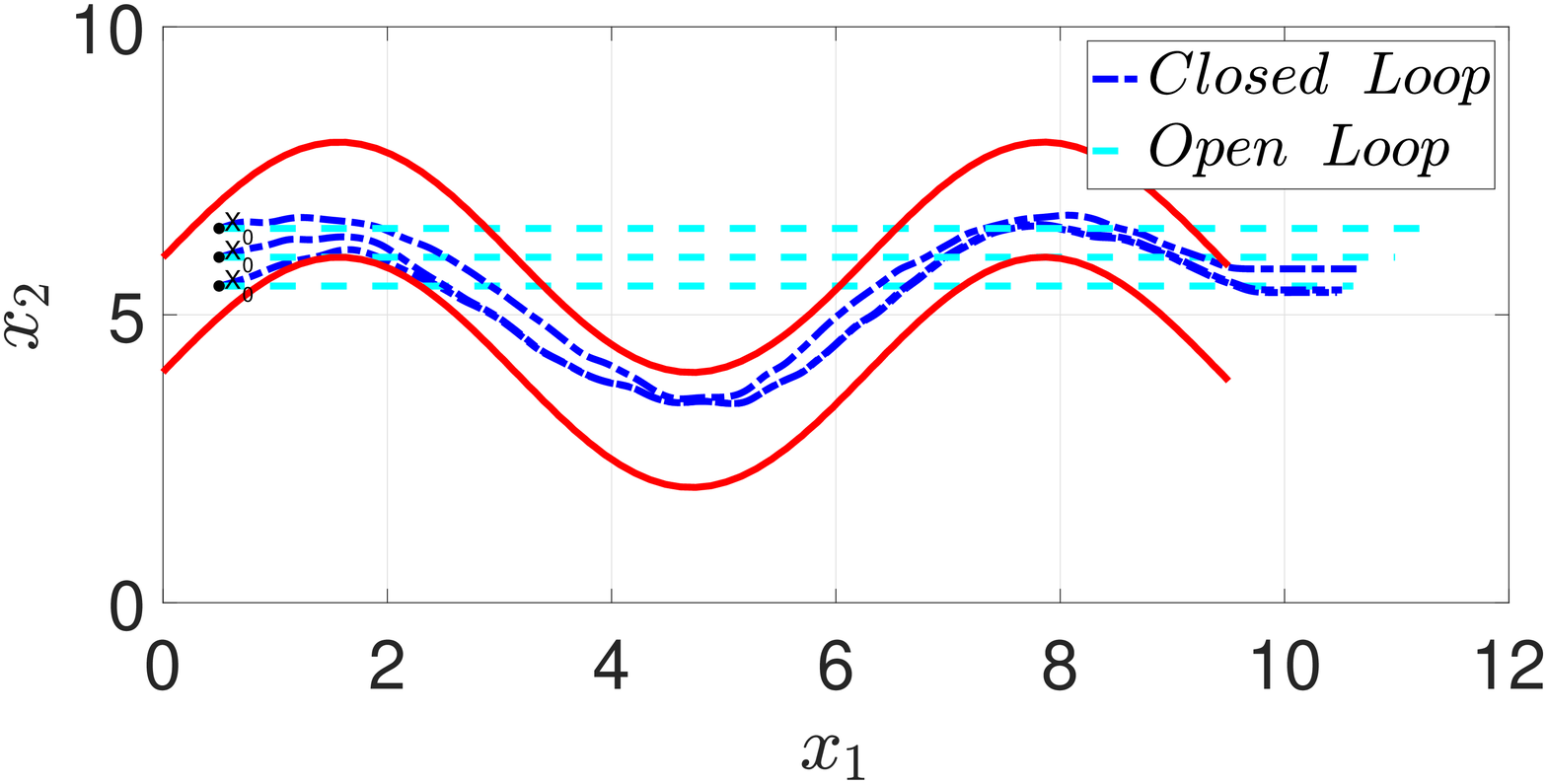}
\caption{Open and Closed loop trajectories navigating the river model prescribed in Example 1. }
\label{exp1_tract}
\end{figure}
\begin{figure}[ht]
\centering
\includegraphics[width=0.9\linewidth]{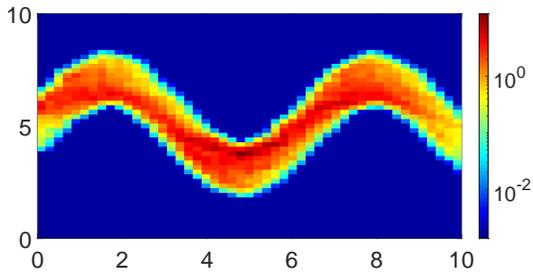}
\caption{Density function obtained from convex optimization in Example 1. }
\label{exp1_rho}
\end{figure}
\noindent{\textbf{Example 2.}} Consider the following controlled dynamics of nonlinear system:
\begin{align}
       \dot{x_1} &= -0.125 + 0.125\cos(0.5x_1) - 0.125\sin(0.5x_2) \nonumber \\
    \dot{x_2} &= u \nonumber 
\end{align}
% The equilibrium point is given by 

% \begin{align}
% x_1&=2\cos^{-1}\left[\sin (0.5 k)+\frac{a-1}{0.125}\right] \nonumber \\
% x_2 &= k \nonumber 
% \end{align}
% Where $k$ is some constant value. We get one of the equilibrium points at $(0,0)$ when $k = 0$ $a=1.125$. We define the sets to be:
For this example, we choose
\begin{itemize}
\small
    \item $\mathbf{X} \triangleq \{\bx \in \mathds{R}^2 : -8 \le x_1 \le 8,-8 \le x_2 \le 8 \}$
    \item $\mathbf{X}_0 \triangleq \{\bx \in \mathbf{X} : 5.3 \le x_1 \le 6.7,-0.7 \le x_2 \le 0.7 \}$
    \item $\mathbf{X}_{u1} \triangleq \{\bx \in \mathbf{X} : 2 \le x_1 \le 4,-3 \le x_2 \le 0 \}$
    \item $\mathbf{X}_{u2} \triangleq \{\bx \in \mathbf{X} : 2 \le x_1 \le 4,0 \le x_2 \le 3 \}$
    \item $\mathbf{X}_{u3} \triangleq \{\bx \in \mathbf{X} : -1 \le x_1 \le 7,-3 \le x_2 \le 3 \}$
    \item $\mathbf{X}_r \triangleq [0,0]^T$
\end{itemize}
$\mathbf{X}_{u3}$ represents the outer deterministic obstacle set where the value of $p(\bx) = 1$ everywhere.  It confines our feasible region to a rectangular block. $\mathbf{X}_{u1}$ and $\mathbf{X}_{u2}$ represents inner obstacle sets with probability measure $\mu_p(\mathbf{X}_{u1})$ and $\mu_p(\mathbf{X}_{u2})$ respectively. They divides our rectangular feasible set into two smaller regions as shown in Fig.~\ref{exp2_scheme}.  We design the problem such that trajectories starting from $\mathbf{X}_0$ finds it's way to $\mathbf{X}_r$ by passing through either $\mathbf{X}_{u1}$ or $\mathbf{X}_{u2}$ depending upon their probability measure. 
%The value of $\gamma$ is taken to be $10^{-13}$
$\Delta t = 0.01$. The probability measure  of $\mathbf{X}_{u1}$ and $\mathbf{X}_{u2}$ are varied to show their effect on the trajectories as can be seen in Fig.~\ref{exp2_tract_0505}-\ref{exp2_tract_100}. The values of density function $\rho(\bx)$ obtained from the convex optimization problem varies with respect to the probability measure of unsafe sets which can be observed in Fig.~\ref{exp2_rho_0505}-\ref{exp2_rho_100}. It is observed that values of $\rho(\bx)$ in the obstacle sets $\mathbf{X}_{u1}$ and $\mathbf{X}_{u2}$ increases when there is a decrease in the probability density function of the obstacle set and vice-versa. This is to be expected from the fact that $\rho(\bx)$ is an occupancy measure and obstacle set with high probability density function will have fewer trajectories entering into them. 
\begin{figure}[ht]
\centering
\includegraphics[width=0.9\linewidth]{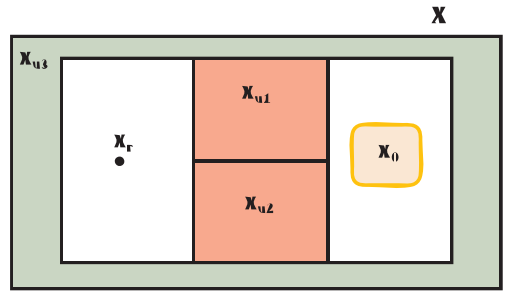}
\caption{Schematic diagram of unsafe sets in Example 2. } 
\label{exp2_scheme}
\end{figure}
\begin{figure}[ht]
\centering
\includegraphics[width=0.9\linewidth]{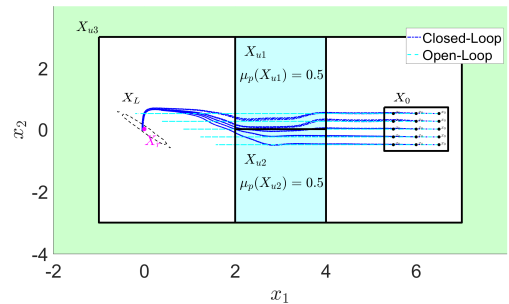}
\caption{State trajectories of Example 2 avoiding unsafe sets $\mathbf{X}_{u1}$ and $\mathbf{X}_{u2}$ both with probability measure of 0.5. } 
\label{exp2_tract_0505}
\end{figure}
\begin{figure}[ht]
\centering
\includegraphics[width=0.9\linewidth]{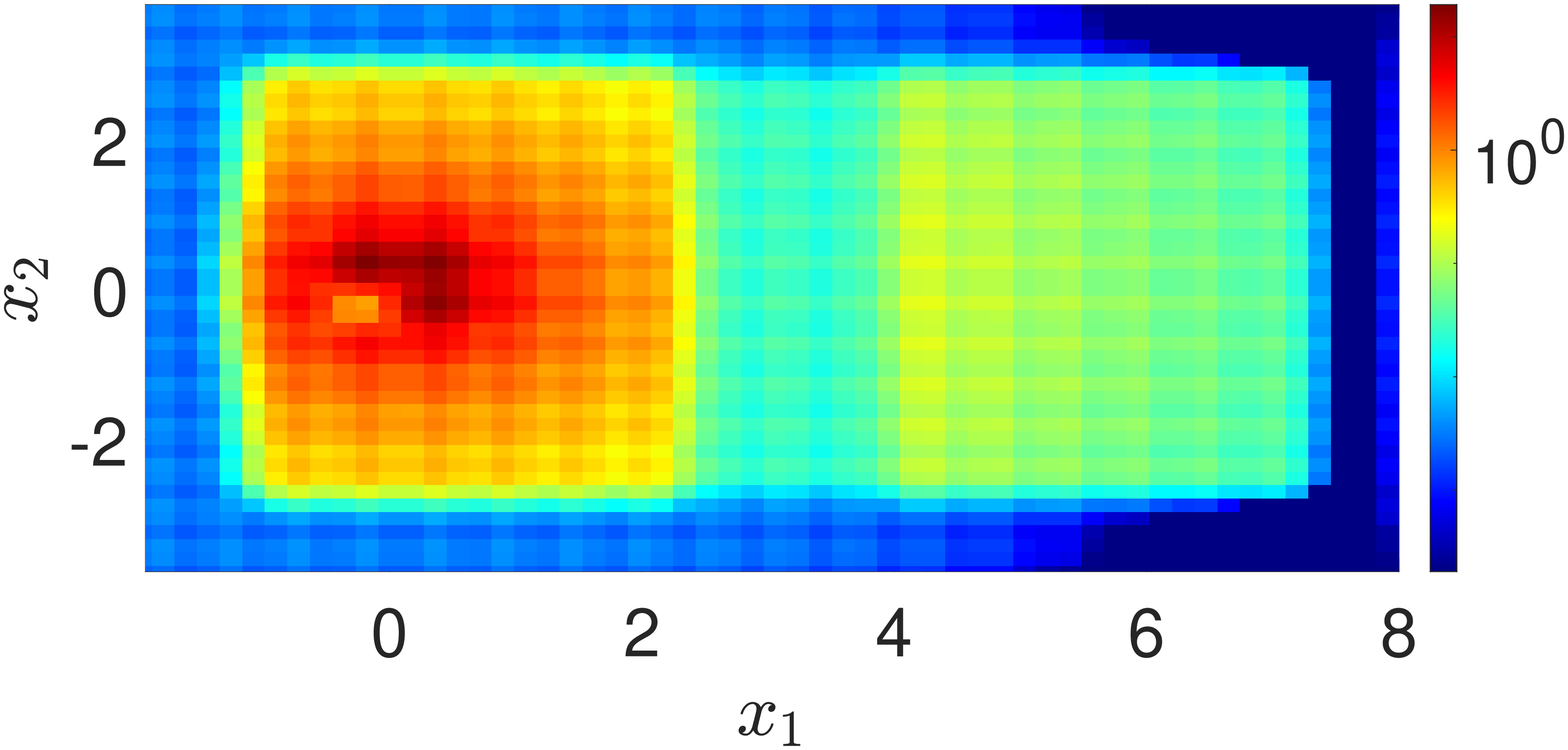}
\caption{$\rho(\bx)$ in Example 2 with respect to probability measure of 0.5 in the unsafe sets $\mathbf{X}_{u1}$ and $\mathbf{X}_{u2}$. } 
\label{exp2_rho_0505}
\end{figure}
\begin{figure}[ht]
\centering
\includegraphics[width=0.9\linewidth]{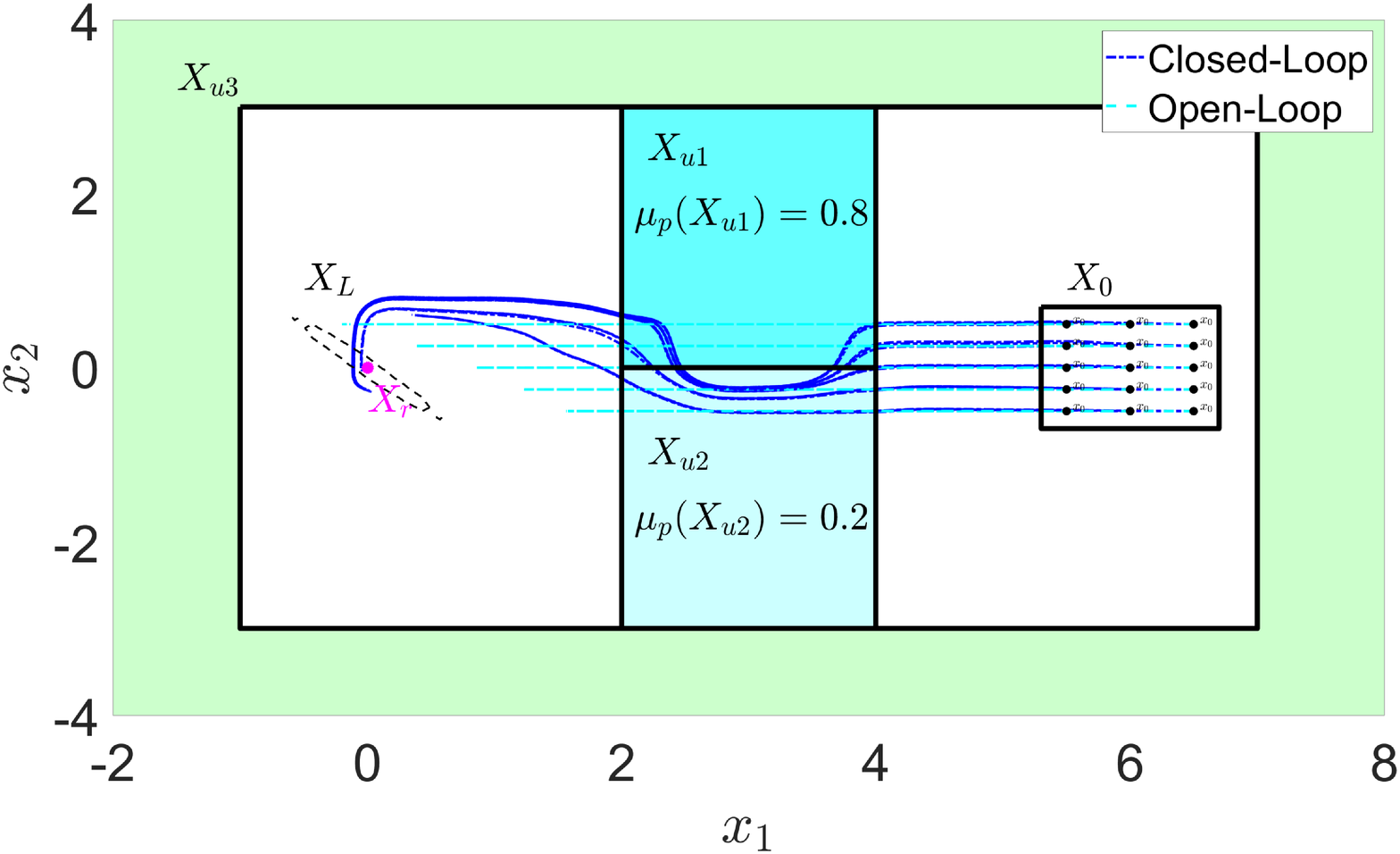}
\caption{State trajectories of Example 2 avoiding unsafe sets $\mathbf{X}_{u1}$ and $\mathbf{X}_{u2}$ with probability measure of 0.8 and 0.2 respectively. } 
\label{exp2_tract_0802}
\end{figure}
\begin{figure}[ht]
\centering
\includegraphics[width=0.9\linewidth]{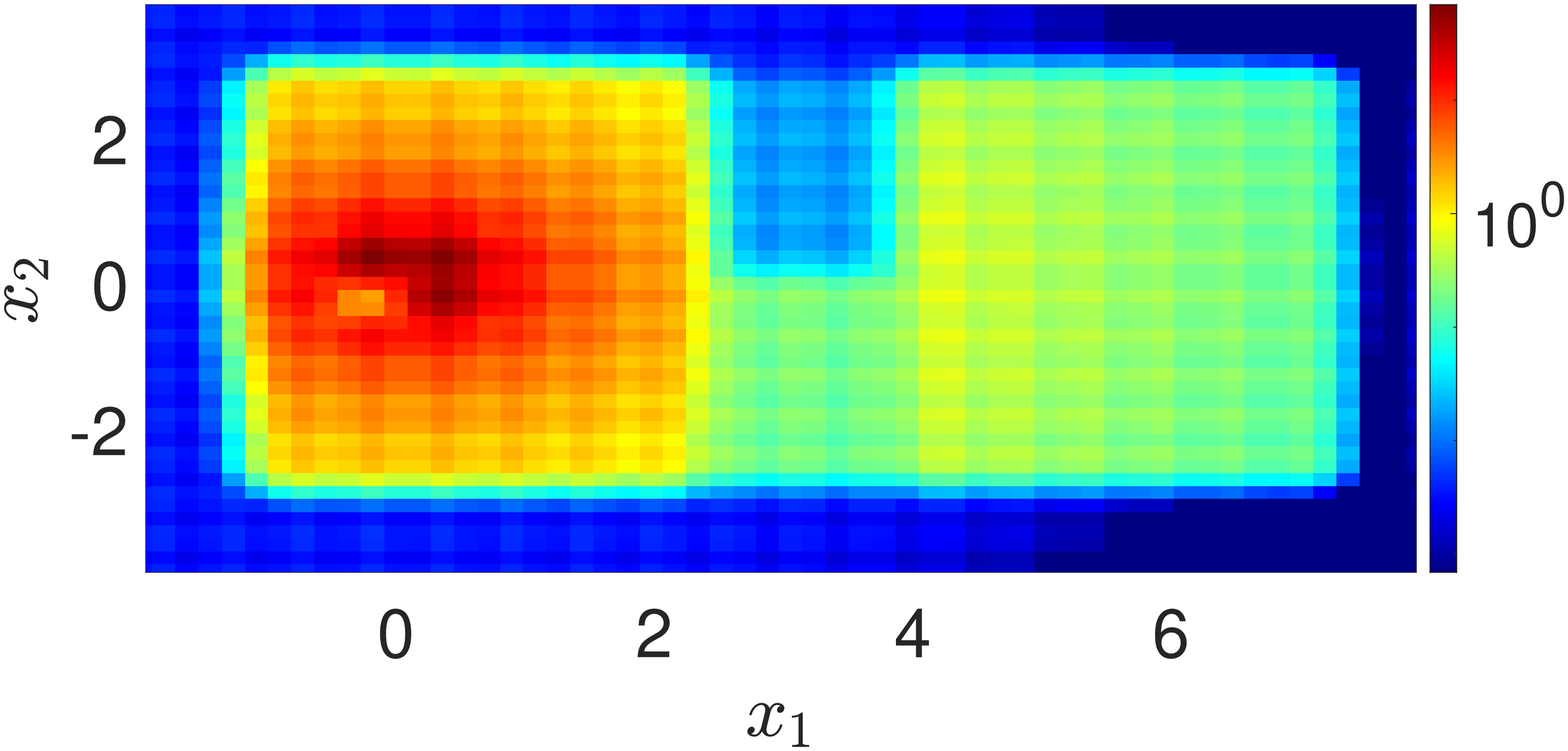}
\caption{$\rho(\bx)$ in Example 2 with respect to probability measure of 0.8 and 0.2 in the unsafe sets $\mathbf{X}_{u1}$ and $\mathbf{X}_{u2}$ respectively. } 
\label{exp2_rho_0802}
\end{figure}
\begin{figure}[ht]
\centering
\includegraphics[width=0.9\linewidth]{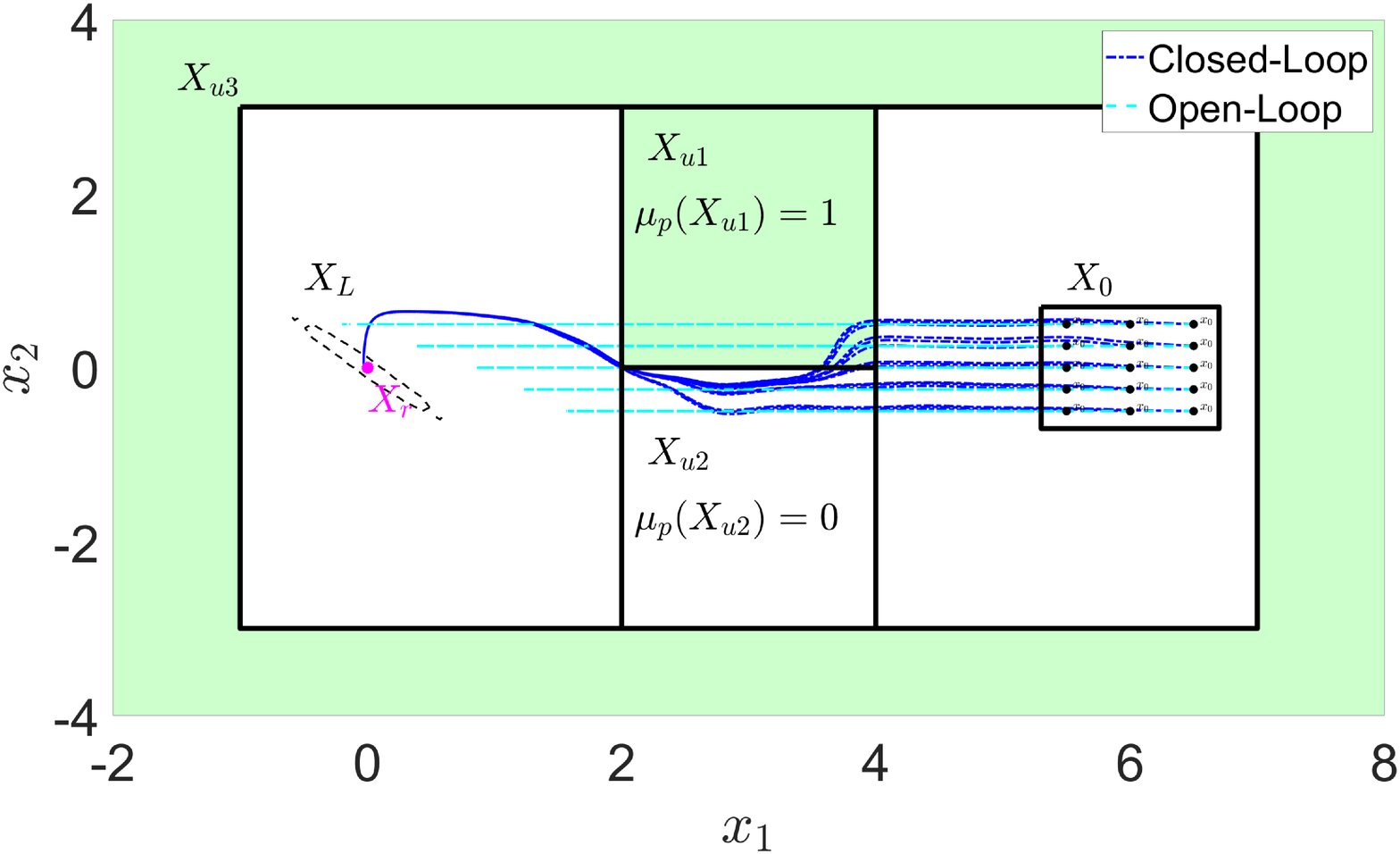}
\caption{State trajectories of Example 2 avoiding unsafe sets $\mathbf{X}_{u1}$ and $\mathbf{X}_{u2}$ with probability measure of 1 and 0 respectively. } 
\label{exp2_tract_100}
\end{figure}
\begin{figure}
\centering
\includegraphics[width=0.9\linewidth]{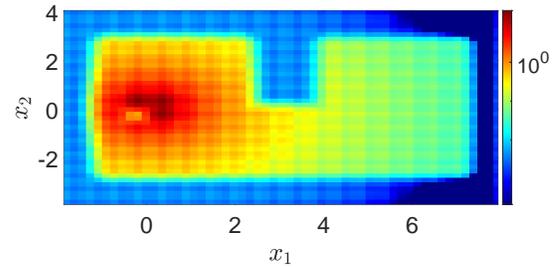}
\caption{$\rho(\bx)$ in Example 2 with respect to probability measure of 1 and 0 in the unsafe sets $\mathbf{X}_{u1}$ and $\mathbf{X}_{u2}$ respectively. } 
\label{exp2_rho_100}
\end{figure}
\section{Conclusions}
We consider the problem of navigation with probabilistic safety constraints. We provide convex formulation to the verification and control synthesis for navigation with probabilistic safety constraints leading to a infinite dimensional convex optimization problem. The finite dimensional approximation of the optimization problem relies on the data-driven approximation of the linear P-F operator. Simulation examples are presented to show the feasibility of the proposed framework. 
% In this paper we consider stochastic navigation problem where the knowledge of the unsafe sets are given through probability density function. We search for a navigation control that satisfies the stochastic safety constraints. By defining the occupancy measure in the density space, we formulate an infinite dimensional convex optimization problem with inequality constraints on the occupancy in the stochastic unsafe set. We then relax the dual of the transformed problem into a finite dimensional convex optimization problem both for stochastic safety control synthesis and stochastic safety verification. Simulations have been conducted to show the feasiblity of the proposed method.

\bibliographystyle{IEEEtran}
\bibliography{ref1,ref}
\end{document}